\documentclass[12pt,oneside,openany]{article}
\usepackage[pdftex]{graphicx}
\usepackage{latexsym}
\usepackage{url}
\usepackage{amsmath}
\usepackage{amssymb} 
\usepackage{makeidx}
\usepackage[top=1in, bottom=1in, left=1in, right=1in]{geometry}

\def\R{{\mathbb R}}

\newtheorem{theorem}{Theorem}   

\newtheorem{prop}[theorem]{Proposition} 
\newtheorem{lemma}[theorem]{Lemma}

\newtheorem{corr}[theorem]{Corollary}
\newtheorem{definition}[theorem]{Definition}

\newenvironment{proof}{\medskip \par \noindent {\bf Proof}\ }{\hfill 
$\Box$ 
                       \medskip \par}

\def\R{{\mathbb R}}

\makeindex

\begin{document}
\overfullrule=0pt
\baselineskip=18pt
\baselineskip=24pt
\font\tfont= cmbx10 scaled \magstep3
\font\sfont= cmbx10 scaled \magstep2
\font\afont= cmcsc10 scaled \magstep2
\title{\tfont Functional Analysis behind a Family of Multidimensional Continued Fractions:\\Part I}
\bigskip
\author{Ilya Amburg\footnote{Center for Applied Mathematics, Cornell University, Ithaca, NY 14853; \texttt{ia244@cornell.edu}}\mbox{  } and Thomas Garrity\footnote{Williams College Department of Mathematics, Williamstown, MA  01267; \texttt{tgarrity@williams.edu}}}

\date{}
\maketitle

\begin{abstract}

Triangle partition maps form a family that includes many, if not most, well-known multidimensional continued fraction algorithms.  This paper begins the exploration of the functional analysis behind the transfer operator of each of these maps. 
We show that triangle partition maps give rise to two classes of transfer operators and present theorems regarding the origin of these classes; we also present related theorems on the form of transfer operators arising from compositions of triangle partition maps. In the next paper,  Part II, we will find eigenfunctions of eigenvalue 1 for transfer operators associated with select triangle partition maps on specified Banach spaces and  then  proceed to prove that the transfer operators, viewed as acting on one-dimensional families of Hilbert spaces, associated with select triangle partition maps are nuclear of trace class zero. We will finish in part II  by deriving Gauss-Kuzmin distributions associated with select triangle partition maps.

\end{abstract}

\section{Introduction}

Using continued fractions to understand the structure of real numbers has been going on for many centuries.  One of the most interesting properties of continued fractions goes back to work of Lagrange in 1789, who showed that a real number has an eventually periodic continued fraction expansion if and only if the number is a quadratic irrational.  This is in analog to the fact that a real number  has an eventually periodic decimal expansion if and only if the number is rational.  As most numbers are neither rational nor quadratics, this leads to the following first of a whole slew of open questions:

\bigskip

\noindent {\bf The Hermite Problem:} Find methods for writing numbers that reflect special algebraic properties.

\bigskip
Even for cubic irrationals, this problem is still quite open.  Attempts to solve it are usually called multidimensional continued fractions.  There are many, seemingly too many. (Multidimensional continued fractions have many other uses.)

Almost all multidimensional continued fraction algorithms associate to a point $(\alpha, \beta)$ in a triangle a  sequence of integers.  For most, eventual periodicity of this sequence means that the $\alpha$ and $\beta$ are in a number field of degree at most three.  What is difficult is the converse.

Most multidimensional continued fraction algorithms can be interpreted as iterative systems on a triangle or, in higher dimensions, on an appropriate simplex.  For a review of many types of multidimensional continued fraction algorithms, see Schweiger's {\it Multidimensional Continued Fractions} \cite{contfrac} and Karpenkov's {\it Geometry of Continued Fractions} \cite{Karpenkov}.  Thinking of these as  iterative systems, it is natural to investigate them as dynamical systems, which in turn will lead to the study of transfer operators (for background on transfer operators, see Baladi  \cite{Baladi00}).  This has long been done for traditional continued fractions, using the Gauss map acting on the unit interval.  For background on this, see  Hensley\cite{ Hensley}, Iosifescu and Kraaikamp \cite{Iosifescu-Kraaikamp1}, Kesseb\"{o}hmer, Munday and Stratmann  \cite{Kessebohmer-Munday-Stratmann},  Khinchin \cite{Khinchin}, Rockett and Szusz \cite{Rockett-Szusz} and Schweiger \cite{Schweiger3}. In \cite{Garrity15}, this was done for the triangle map, where it was shown that much of the work of Mayer and Roepstorff \cite{Mayer, Mayer-Roepstorff1988} has nontrivial analogs for the triangle map. Building on an earlier version of \cite{Garrity15} , there is also the interesting work of Bonanno, Del Vigna and Munday \cite{Bonanno-Del Vigna-Munday}.  Further, there is the work of Fourgeron and Skripchenko \cite{Fougeron-Skripchenko} on the Lyapunov exponents of the triangle map.   The goal of this paper is to see which of these analogs hold for various members of the family of triangle partition maps.

Triangle partition maps \cite{SMALLTRIP, Cubic, Karpenkov} are a family of $216$ multidimensional continued fraction algorithms that include (when combinations are allowed) many, if not most, well-known multidimensional continued fraction algorithms, which is why they are a natural class of algorithms to study.  These maps are reviewed in section \ref{background}.  

In section \ref{transferoperators}, we will start to look at transfer operators associated with triangle partition maps.  For the Gauss map, we know that 
the transfer operator is 
$$\mathcal{L}(f)(x) = \sum_{k=1}^{\infty} \frac{1}{(k+x)^2} f\left( \frac{1}{(k+x)}   \right),$$
and for the triangle map, the transfer operator is the somewhat similar looking,
$$\mathcal{L}(f)(x,y) = \sum_{k=0}^{\infty} \frac{1}{(1+kx + y)^3} f\left( \frac{1}{(1 + k x + y) } , \frac{x}{(1 + k x + y) }  \right).$$

In section \ref{poly}, 
we will show  for precisely half of the triangle partition maps that the transfer operator will have the form
$$\sum_{k=0}^{\infty} \frac{1}{m(k,x,y)^3} f\left(\cdot, \cdot \right)$$
where $m(k,x,y)$ will be a polynomial linear in $k$ (with the possibility of a $(-1)^k$ term) and linear in $x$ and $y$.  We will say that such transfer operators have {\it polynomial behavior}.  We will then show  for the remaining half of triangle partition maps
that their transfer operators have drastically different behavior. We will see that their transfer operators are of the general form
$$\sum_{k=0}^{\infty} \frac{1}{e(k,x,y)^3} f\left(\cdot, \cdot \right)$$
where $e(k,x,y)$ is exponential in $k$.  Such triangle partition algorithms are naturally enough said to have {\it non-polynomial behavior}.  It is here where this paper (Part I) ends.

But this paper sets up, more importantly,  the needed machinery for the next paper ``Functional Analysis behind a Family of Multidimensional Continued Fractions: Part II''  \cite{Amburg-Garrity II} In this next paper we 
turn to the spectrum of the transfer operators.  Here we need to be concerned with which vector spaces of functions we are considering.  We will see in  section two of that paper that for some triangle partition maps there are natural Banach spaces of functions, that for $18$ of these  the largest eigenvalue of the associated transfer operator is one, and that for some of these  the dimension of the corresponding eigenspace is one.   Then we will see in section three that a number of the triangle partition maps' transfer operators are nuclear operator of trace class zero when viewed as acting on naturally occurring one-dimensional families of Hilbert spaces.  In section four we present the Gauss-Kuzmin statistics for select triangle partition maps. In the conclusion of this next paper  we discuss future directions for work.

We would like to thank the referees for their suggestions.

\section{Background on Triangle Partition Maps}\label{background}
As this is background, much of this section is similar to certain  sections in \cite{stern-trip, GarrityT05, SMALLTRIP, Cubic, Triangle, Karpenkov, Schweiger05}.

\subsection{The Triangle Map}

We start with the triangle map,  from which, in the next section,  the  216 triangle partition maps are constructed. For a further explanation of the triangle map, see \cite{Triangle, GarrityT05}.

Partition the triangle  
$$ \bigtriangleup = \{ (x,y): 1 \geq x \geq y \geq  0 \}$$
into subtriangles
$$\bigtriangleup_{k} = \{(x,y)\in\bigtriangleup:1-x-ky 
\geq 0 > 1-x-(k+1)y\},$$
for each non-negative integer $k$.  The partitioning is represented in the following diagram:

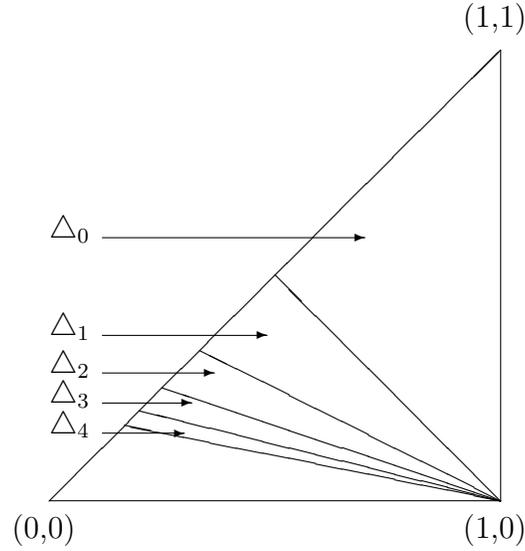
\begin{center}
\setlength{\unitlength}{.1 cm}
\begin{picture}(70,70)
\put(5,5){\line(1,0){60}}
\put(65,5){\line(0,1){60}}
\put(5,5){\line(1,1){60}}
\put(0,0){(0,0)}
\put(60,0){(1,0)}
\put(60,68){(1,1)}

\put(65,5){\line(-1,1){30}}
\put(65,5){\line(-2,1){40}}
\put(65,5){\line(-3,1){45}}
\put(65,5){\line(-4,1){48}}
\put(65,5){\line(-5,1){50}}

\put(5,40){$\triangle_{0}$}
\put(12,40){\vector(1,0){35}}
\put(5,27){$\triangle_{1}$}
\put(12,27){\vector(1,0){22}}
\put(5,22){$\triangle_{2}$}
\put(12,22){\vector(1,0){15}}
\put(5,18){$\triangle_{3}$}
\put(12,18){\vector(1,0){12}}
\put(5,14){$\triangle_{4}$}
\put(12,14){\vector(1,0){11}}
\end{picture}
\end{center}
Each subtriangle $\triangle_k$ has vertices $(1,0), (1/(k+1), 1/(k+1)), (1/(k+2), 1/(k+2)).$  The triangle map
$T:\triangle \rightarrow \triangle$ is defined by setting, for any $(x,y)\in \triangle_k$,
$$T(x,y) = \left(\frac{y}{x}, \frac{1 - x - k y}{x}\right).$$
This is an analog of the traditional Gauss map for continued fractions.

And as with the Gauss map and continued fractions, it is useful to translate this into matrices.
We have the projection map
$$\pi(x,y,z) = \left( \frac{y}{x}, \frac{z}{x} \right)$$
mapping rays in $\R^3$ to points in $\R^2$.  This allows us to associate to the triangle $\triangle $ in $\R^2$ the cone in $\R^3$ (which we will also denote by $\triangle$)
$$\triangle = \left\{  a \left( \begin{array}{c} 1\\0\\0   \end{array}  \right) + b \left( \begin{array}{c} 1\\1\\0   \end{array}  \right) +c\left( \begin{array}{c} 1\\1\\1   \end{array}  \right) :a,b,c>0 \right\}.$$
The cone $\triangle$ is the positive span of the column vectors of the matrix
$$V= (v_1, v_2, v_3) = \left( \begin{array}{ccc} 1&1&1\\0&1&1\\0  &0&1 \end{array}  \right).$$
Define matrices 

$$F_0 = \left(\begin{array}{ccc}
0 & 0 & 1 \\ 
1 & 0 & 0 \\
0 & 1 & 1 \\
\end{array}\right),
F_1 = \left(\begin{array}{ccc}

1 & 0 & 1 \\ 
0 & 1 & 0 \\
0 & 0 & 1 \\
\end{array}\right)$$
These matrices occur naturally since the cone corresponding to the subtriangles $\triangle_k$ is the positive span of the columns of $VF_1^{k-1}F_0$, which by an abuse of notation, we will often write as 
$$\triangle_k = VF_1^{k-1}F_0.$$

Then the triangle map $T: \triangle \rightarrow \triangle$ can be shown to be equal to:
 $$T(x, y) = \pi\left((1, x, y) \left(V F_0^{-1} F_1^{-k} V^{-1} \right)^{T}\right)$$
if $(x,y) \in \bigtriangleup_k.$

This enables us to associate the {\it triangle sequence} $(a_0, a_1, ... )$ to a point $(a,b) \in \bigtriangleup$ by letting $a_i$ equal $k$ if $T^i(x,y) \in \bigtriangleup_k$.  If  for any $k$ we 
have $T^{k}(a, b) \in \{ (x,0): 0 \leq 
 x\leq 1 \}$, the sequence terminates.

\subsection{ Triangle Partition Maps: TRIP maps }

The above triangle map depends on an initial choice of three vertices for $\triangle$.  If we permute these vertices both before and after applying the matrices $F_0$ and $F_1$, we create different multidimensional continued fraction algorithms.  This is at the heart of the definition for the 216 triangle partition maps. 

More precisely, we will allow permutations of the initial vertices by some $\sigma \in S_3,$ by some $\tau_1 \in S_3$ after applying $F_1,$ and by some $\tau_0 \in S_3$ after applying $F_0.$ 

This leads us to define the matrices
\begin{eqnarray*}
F_0(\sigma, \tau_0,  \tau_1) &=&  \sigma F_0 \tau_0,\\
F_1(\sigma, \tau_0,  \tau_1) &=& \sigma F_1 \tau_1
 \end{eqnarray*}
for every $(\sigma, \tau_0, \tau_1) \in S_3^3.$

In particular, we denote the permutation matrices as follows:
\bigskip

$e=\left(
\begin{array}{ccc}
 1 & 0 & 0 \\
 0 & 1 & 0 \\
 0 & 0 & 1 \\
\end{array}
\right),$
$(12)=\left(
\begin{array}{ccc}
 0 & 1 & 0 \\
 1 & 0 & 0 \\
 0 & 0 & 1 \\
\end{array}
\right),$
$(13)=\left(
\begin{array}{ccc}
 0 & 0 & 1 \\
 0 & 1 & 0 \\
 1 & 0 & 0 \\
\end{array}
\right),$\\

\bigskip
$(23)=\left(
\begin{array}{ccc}
 1 & 0 & 0 \\
 0 & 0 & 1 \\
 0 & 1 & 0 \\  
\end{array}
\right),$
$(123)=\left(
\begin{array}{ccc}
 0 & 1 & 0 \\
 0 & 0 & 1 \\
 1 & 0 & 0 \\
\end{array}
\right),$ and
$(132)=\left(
\begin{array}{ccc}
 0 & 0 & 1 \\
 1 & 0 & 0 \\
 0 & 1 & 0 \\
\end{array}
\right).$
\bigskip

Let $\bigtriangleup_k(\sigma, \tau_0,  \tau_1)$ be the image of $\bigtriangleup$ under the action of $F_1^k(\sigma, \tau_0,  \tau_1) F_0(\sigma, \tau_0,  \tau_1).$  Thus, with an abuse of notation, we have 
$$\bigtriangleup_k(\sigma, \tau_0,  \tau_1)= VF_1^k(\sigma, \tau_0,  \tau_1) F_0(\sigma, \tau_0,  \tau_1).$$
Then define a {\it triangle partition map} as  $T_{\sigma, \tau_0, \tau_1}: \bigcup_{k=0}^{\infty}\bigtriangleup_k(\sigma, \tau_0,  \tau_1) \to \bigtriangleup$ where 
$$T_{\sigma, \tau_0, \tau_1}(x, y) = \pi\left((1, x, y) \left(V F_0^{-1}(\sigma, \tau_0,  \tau_1) F_1^{-k}(\sigma, \tau_0,  \tau_1) V^{-1} \right)^{T}\right)$$
if $(x,y)\in\bigtriangleup_k(\sigma, \tau_0,  \tau_1).$

As $S_3$ has six elements, we have $6^3 = 216$ different triangle partition maps.
Of course, the original triangle map corresponds to $T_{e,e,e}.$  Appendix A in Part II gives the explicit forms of $T_{\sigma, \tau_0, \tau_1}$ for $108$ of the maps.  One of the goals of this paper is to explain why we do not list in an easy-to-read table the forms of $T_{\sigma, \tau_0, \tau_1} $ for the remaining $108$ maps.

We obtain an even larger family of maps by allowing compositions of triangle partition maps \cite{SMALLTRIP}. As an example, we might perform the first subdivision of $\bigtriangleup$ using $(\sigma, \tau_0, \tau_1),$ the second subdivision using $(\sigma_2, \tau_{02}, \tau_{12})$ and so on. We can represent such compositions as $T_1 \circ T_2 \ldots \circ T_n,$ where, of course, each subscript is short for a permutation triplet. 

Again, many, if not most, multidimensional continued fraction algorithms can be put into this language \cite{SMALLTRIP}, which is why these are natural maps to consider.  To get a feel of these maps, we consider two examples in a bit more detail.

\subsection{The $((23), (23), (23))$ Case}\label{((23), (23), (23))}
This is simply an example of what we will later call a polynomial-behavior triangle partition map.  It is in and of itself not that significant.  
We have 
\begin{eqnarray*}
F_0 ((23), (23), (23)) &=& (23) F_0 (23) \\
&=& \left(\begin{array}{ccc}1 & 0 & 0 \\
 					0 & 0 & 1 \\
 					0 & 1 & 0 \\  
					\end{array}\right) 
	 \left(\begin{array}{ccc}
 					0 & 0 & 1 \\
 					1 & 0 & 0 \\
 					0 & 1 & 1 \\  
					\end{array}\right) 
	\left(\begin{array}{ccc}
 					1 & 0 & 0 \\
 					0 & 0 & 1 \\
					 0 & 1 & 0 \\  \end{array}\right)  \\
&=& \left(
\begin{array}{ccc}
 0 & 1 & 0 \\
 0 & 1 & 1 \\
 1 & 0 & 0 \\  
\end{array}
\right),\\
F_1 ((23), (23), (23)) &=& (23) F_1 (23) \\
&=& \left(\begin{array}{ccc}
 					1 & 1 & 0 \\
 					0 & 1 & 0 \\
					 0 & 0 & 1 \\  \end{array}\right) 
\end{eqnarray*}

We want to see what the subtriangles $\triangle_k ((23), (23), (23))$ look like.  We have 
\begin{eqnarray*}
\bigtriangleup_k  ((23), (23), (23)) &=& V F_1^k((23), (23), (23)) F_0((23), (23), (23))\\
&=&  \left(\begin{array}{ccc}
 					1 & 1 & 1 \\
 					0 & 1 & 1 \\
 					0 & 0 & 1 \\  
					\end{array}\right)   \left(\begin{array}{ccc}
 					1 & 1 & 0 \\
 					0 & 1 & 0 \\0&0&1 \end{array}\right)^k   \left(
\begin{array}{ccc}
 0 & 1 & 0 \\
 0 & 1 & 1 \\
 1 & 0 & 0 \\  
\end{array}
\right)  \\
&=&\left(\begin{array}{ccc}
 					1 & 1 & 1 \\
 					0 & 1 & 1 \\
 					0 & 0 & 1 \\  
					\end{array}\right)   \left(\begin{array}{ccc}
 					1 & k & 0 \\
 					0 & 1 & 0 \\0&0&1 \end{array}\right)   \left(
\begin{array}{ccc}
 0 & 1 & 0 \\
 0 & 1 & 1 \\
 1 & 0 & 0 \\  
\end{array}
\right)  \\
&=& \left(\begin{array}{ccc}
 1 & k+2 & k+1 \\
 1 & 1 & 1 \\
 1 & 0 & 0 \\  
\end{array}
\right) 
\end{eqnarray*} 

Pictorially, we have

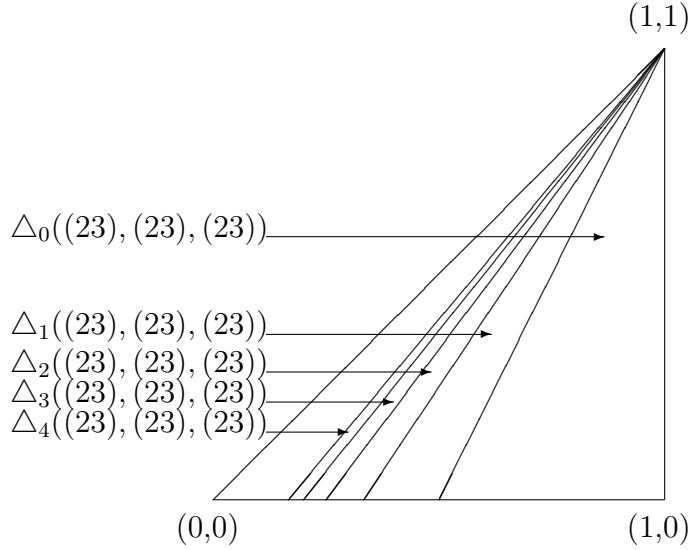
\begin{center}
\setlength{\unitlength}{.1 cm}
\begin{picture}(70,70)
\put(5,5){\line(1,0){60}}
\put(65,5){\line(0,1){60}}
\put(5,5){\line(1,1){60}}

\put(0,0){(0,0)}
\put(60,0){(1,0)}
\put(60,68){(1,1)}

\put(65,65){\line(-1,-2){30}}
\put(65,65){\line(-2,-3){40}}
\put(65,65){\line(-3, -4){45}}
\put(65,65){\line(-4, -5){48}}
\put(65,65){\line(-5, -6){50}}

\put(-22,40){$\triangle_{0}((23), (23), (23))$}
\put(12,40){\vector(1,0){45}}
\put(-22,27){$\triangle_{1}((23), (23), (23))$}
\put(12,27){\vector(1,0){30}}
\put(-22,22){$\triangle_{2}((23), (23), (23))$}
\put(12,22){\vector(1,0){22}}
\put(-22,18){$\triangle_{3}((23), (23), (23))$}
\put(12,18){\vector(1,0){17}}
\put(-22,14){$\triangle_{4}((23), (23), (23))$}
\put(12,14){\vector(1,0){11}}

\end{picture}
\end{center}
In the plane $\R^2$, the vertices of each $\bigtriangleup_k  ((23), (23), (23)) $ are 
$$\left\{\left( \frac{1}{1}, \frac{1}{1}  \right), \left( \frac{1}{k+2}, \frac{0}{1}  \right), \left( \frac{1}{k+1}, \frac{0}{1}  \right)\right\}= \left\{\left(1,1  \right), \left( 1/(k+2), 0  \right), \left(1/(k+1), 0 \right)\right\}$$

Now,  for any $(x,y) \in \bigtriangleup_k  ((23), (23), (23)) $, we have
\begin{eqnarray*}
T_{\sigma, \tau_0, \tau_1}(x, y) &=& \pi\left((1, x, y) \left(V F_0^{-1}(\sigma, \tau_0,  \tau_1) F_1^{-k}(\sigma, \tau_0,  \tau_1) V^{-1} \right)^{T}\right)  \\
&=& \left(\frac{x-y}{x},\frac{-1 + (2 + k) x + (- 1 - k) y}{x} \right),
\end{eqnarray*}
in which case,
$$T_{23,23,23}^{-1}=\left(\frac{1}{1+(1+k)x-y}, \frac{1-x}{1+(1+k)x-y}\right),$$
a formula that we will need later.

\subsection{  The $(e, e, (13))$ Case }   \label{(e, e, (13))}

Here we present another example of a TRIP map.  We will see that its behavior is quite different than that of  TRIP maps $(e,e,e)$ and $((23), (23), (23)),$ a difference that will be explained in section \ref{poly}.

We first need to calculate $F_0(e, e, (13))$ and $F_1(e, e, (13))$.  $F_0(e, e, (13))$ is easy, as
$$F_0(e, e, (13)) = e F_0 e =  \left(\begin{array}{ccc}
 					0 & 0 & 1 \\
 					1 & 0 & 0 \\
 					0 & 1 & 1 \\  
					\end{array}\right)   .$$
For the other matrix, we have 
\begin{eqnarray*}
F_1(e,e,(13)) & = & eF_1  (13) \\
&=&    \left( \begin{array}{ccc} 1 & 0 & 1  \\ 0 & 1 & 0 \\ 0 & 0 & 1\end{array}  \right)   \left( \begin{array}{ccc} 0 & 0 & 1  \\ 0 & 1 & 0 \\ 1 & 0 & 0\end{array}  \right)\\
&=& \left( \begin{array}{ccc} 1 & 0 & 1  \\ 0 & 1 & 0 \\ 1 & 0 & 0\end{array}  \right) 
\end{eqnarray*}
 We need to find a clean formula for $F_1^k(e,e,(13)) .$  We will need 
 the Fibonacci numbers, which are given by 
 $f_{k+1} = f_{ k } + f_{ k-1 },$
 with initial terms
 $f_{-1}=0, f_0 = 1.$  Note that this means that $f_{-2} = 1.$
 
\begin{prop}
$$F_1^k(e,e,(13)) =  \left( \begin{array}{ccc} f_{ k } & 0 & f_{ k-1 }  \\ 0 & 1 & 0 \\   f_{  k-1  } & 0 & f_ {k-2}\end{array}  \right)  $$
\end{prop}

\begin{proof} By induction.  For $k=0$, we have that $F_1^0(e,e,(13))$ is the identity. But for $k=0$, we have 
$$ \left( \begin{array}{ccc} f_{ k } & 0 & f_{ k-1 }  \\ 0 & 1 & 0 \\   f_{  k-1  } & 0 & f_ {k-2}\end{array}  \right) = \left( \begin{array}{ccc} f_{ 0 } & 0 & f_{ -1 }  \\ 0 & 1 & 0 \\   f_{  -1  } & 0 & f_ {-2}\end{array}  \right) = \left( \begin{array}{ccc} 1 & 0 & 0  \\ 0 & 1 & 0 \\   0 & 0 &1\end{array}  \right).$$

Then in general we have 
\begin{eqnarray*}
F_1(e,e,(13))^k   &=&  F_1(e,e,(13))^{k-1}   F_1(e,e,(13))  \\
&=&    \left( \begin{array}{ccc} f_{ k-1 } & 0 & f_{ k-2 }  \\ 0 & 1 & 0 \\   f_{  k-2  } & 0 & f_ {k-3}\end{array}  \right)  \left( \begin{array}{ccc} 1 & 0 & 1  \\ 0 & 1 & 0 \\ 1 & 0 & 0\end{array}  \right) \\
&=&    \left( \begin{array}{ccc} f_{ k-1 }  + f_{k-2} & 0 & f_{ k-1 }  \\ 0 & 1 & 0 \\   f_{  k-2  } + f_{k-3}& 0 & f_ {k-2}\end{array}  \right) \\
&=&  \left( \begin{array}{ccc} f_{ k } & 0 & f_{ k-1 }  \\ 0 & 1 & 0 \\   f_{  k-1  } & 0 & f_ {k-2}\end{array}  \right) ,
\end{eqnarray*}
as desired.

\end{proof}

Now to find the subtriangles $\triangle_k(e,e,(13)) = VF_1(e,e,(13))^{k}F_0(e,e,(13))$.

 \begin{prop}
 $$\triangle_k(e,e,(13)) =V F_1(e,e,(13))^{k} F_0(e,e,(13)) =    \left( \begin{array}{ccc} 1& f_{ k } & f_{ k + 2 }  \\ 1 & f_{k-2} & f_{k} \\   0& f_{ k-2 } & f_ {k}\end{array}  \right).$$
 \end{prop}
 \begin{proof}  By calculation this is
 \begin{eqnarray*}
 \triangle_k(e,e,(13))&=&  V F_1(e,e,(13))^{k} F_0(e,e,(13)) \\
  &=&   \left( \begin{array}{ccc} 1& f_{ k } & f_{ k + 2 }  \\ 1 & f_{k-2} & f_{k} \\   0& f_{ k-2 } & f_ {k}\end{array}  \right)  .
 \end{eqnarray*}

 \end{proof}

Thus the subtriangle $ \triangle_k(e,e,(13))$ in the plane $\R^2$ has vertices
$$(1,0), ( f_{k-2}/f_{k }, f_{k-2}/f_{k } ), ( f_{k}/f_{k+2 }, f_{k}/f_{k + 2 } ).$$

Pictorially, we have

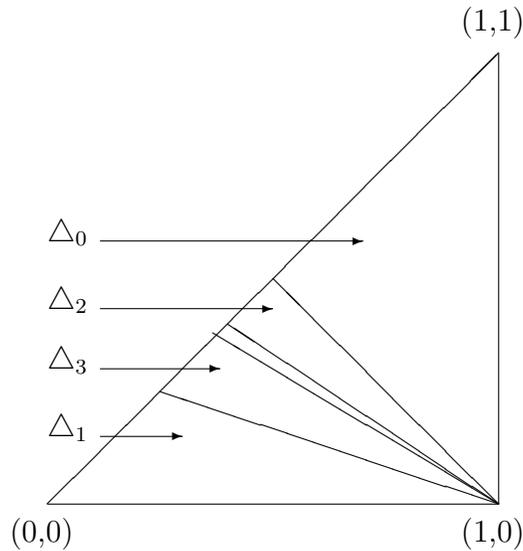
\begin{center}
\setlength{\unitlength}{.1 cm}
\begin{picture}(70,70)
\put(5,5){\line(1,0){60}}
\put(65,5){\line(0,1){60}}
\put(5,5){\line(1,1){60}}
\put(0,0){(0,0)}
\put(60,0){(1,0)}
\put(60,68){(1,1)}

\put(65,5){\line(-1,1){30}}
\put(65,5){\line(-3,1){45}}
\put(65,5){\line(-3,2){36}}
\put(65,5){\line(-5,3){38}}

\put(5,40){$\triangle_{0}$}
\put(12,40){\vector(1,0){35}}

\put(5,14){$\triangle_{1}$}
\put(12,14){\vector(1,0){11}}

\put(5,31){$\triangle_{2}$}
\put(12,31){\vector(1,0){23}}

\put(5,23){$\triangle_{3}$}
\put(12,23){\vector(1,0){16}}

\end{picture}
\end{center}

We can now calculate the TRIP map, whose proof is just a somewhat painful calculation.

\begin{prop} For $(x,y) \in \triangle_k(e,e,(13))$, the TRIP map
$$T(e,e,(13)) (x,y) = \pi \left[ (1,x,y) \left(  V (VF_1(e,e,(13))^{k}F_0(e,e,(13))^{-1}    \right)^T  \right]$$
has $x$ coordinate
$$    	\frac{     (-1)^{k+1} f_{k+1}  +       (-1)^{k+2} f_{k+1} x   +  (-1)^{k+2} f_{k+2} y   }{     (-1)^{k+1} f_{k+1}  + ( 1- (-1)^{k+1} f_{k+1} )x   +  ( -1    + (-1)^{k+2} f_{k+2})y        }$$
 and $y$-coordinate
 $$\frac{      (-1)^{ k-1 } f_{ k - 1 } -   (-1)^{ k-1 } f_{ k - 1 } x +(-1)^{ k } f_{ k    } y          }{     (-1)^{k+1} f_{k+1}  + ( 1- (-1)^{k+1} f_{k+1} )x   +  ( -1    + (-1)^{k+2} f_{k+2})y        }	.					.$$
\end{prop}

The fact that we had to use two lines to write out the TRIP map $T(e,e,(13)) (x,y) $ is not just chance, as we will see in section \ref{poly}.

\section{Transfer Operators}
\label{transferoperators}
For general background on transfer operators and their importance, see Baladi \cite{Baladi00}.

 In general, for a dynamical system corresponding to the map $T: X \to X,$ a transfer operator linearly maps functions from a vector space of functions on $X$ to another (not necessarily different) vector space of functions on $X.$ For a more concrete definition, define a function $g: X \to \R.$ Then the transfer operator $\mathcal{L}_T$ acting on $f: X \to \R$ is defined by 
$$\mathcal{L}_T f(x) = \sum_{y:T(y)=x} g(y)f(y). $$
If $T$ is piece-wise differentiable, as for our maps, we choose $g = \frac{1}{\left|Jac(T)\right|}$, as is traditional and natural.  

The spectrum for each transfer operator yields much information about the algorithm.  One of the main motivations is that if the transfer operator has a leading eigenfunction with eigenvalue one with multiplicity one, we expect that this eigenfunction will give rise to the natural invariant measure for the map $T$.  This is exactly what Gauss conjectured for the transfer operator for the traditional continued fraction, conjectures only proven in the 1920s by Kuzmin.  (Again, the goal of this paper is to start such investigations for TRIP maps, and hence for almost all multi-dimensional continued fraction algorithms.)

We define our transfer operators to be 
$$\mathcal{L}_{T_{\sigma, \tau_0,\tau_1}}f(x,y) = \sum_{(a,b):T_{\sigma, \tau_0,\tau_1}(a,b)=(x,y)} \frac{1}{\left|\mbox{Jac}(T_{\sigma, \tau_0,\tau_1}(a,b))\right|} f(a,b).$$
While a long calculation is required to arrive at the transfer operator corresponding to each $(\sigma, \tau_0,\tau_1),$ the calculations are not difficult theoretically and have been carried out for each of the 216 triangle partition  maps. As an example, as shown in \cite{Garrity15},
$$\mathcal{L}_{T_{e,e,e}}f(x,y) = \sum_{k=0}^{\infty} \frac{1}{(1+kx + y )^3} f\left( \frac{1}{1+kx+ y},  \frac{x}{1+kx+y} \right).$$

We have calculated the explicit form of $$\mathcal{L}_{T_{\sigma, \tau_0,\tau_1}}f(x,y) = \sum_{(a,b):T_{\sigma, \tau_0,\tau_1}(a,b)=(x,y)} \frac{1}{\left|\mbox{Jac}(T_{\sigma, \tau_0,\tau_1}(a,b))\right|} f(a,b)$$ for all $(\sigma,\tau_0,\tau_1)\in S_{3}^{3}.$ As mentioned in the introduction, one of the main goals for this paper is to show that the family of 216 triangle partition maps naturally splits into two classes: 108 maps exhibiting what we term polynomial behavior, and 108 maps exhibiting  non-polynomial behavior. Polynomial-behavior maps' transfer operators have compact explicit forms while non-polynomial-behavior maps' transfer operators have very lengthy and complex forms.  The  explicit forms for all of these maps are available online (see \cite{Amburg}).  

As another example, let us return to the $((23), (23), (23))$ TRIP map of subsection \ref{((23), (23), (23))}. This has to  have polynomial behavior, as otherwise it would be too cumbersome to express.
We need to find  the Jacobian of $T_{23,23,23}$ with any $x$ replaced by the first component of $T_{23,23,23}^{-1}$ and any $y$ replaced by the second component of $T_{23,23,23}^{-1}.$ 
The Jacobian is:
\begin{eqnarray*}
\mbox{Jac}(T_{23,23,23})
&=& 
\det
\begin{pmatrix}
\frac{\partial}{\partial x} \left(\frac{x-y}{x} \right)  & \frac{\partial}{\partial y} \left(\frac{x-y}{x} \right)\\
\frac{\partial}{\partial x} \left( \frac{-1 + (2 + k) x + (- 1 - k) y}{x}\right)  & \frac{\partial}{\partial y} \left(\frac{-1 + (2 + k) x + (- 1 - k) y}{x} \right)\\
\end{pmatrix}\\
&=&
\det
\begin{pmatrix}
\frac{y}{x^2} & \frac{-1}{x} \\
\frac{1+y+ky}{x^2} & \frac{-(1+k)}{x} \\
\end{pmatrix}\\
&=&1/x^3
\end{eqnarray*}
Substituting the first component of $T_{23,23,23}^{-1}$ for $x,$ the Jacobian becomes $1/(\frac{1}{1+(1+k)x-y})^3=(1+(1+k)x-y)^3,$ so that 
$$\frac{1}{\mbox{Jac}(T_{23,23,23})}=\frac{1}{(1+(1+k)x-y)^3}.$$

Hence
\begin{eqnarray*}
\mathcal{L}_{T_{23,23,23}}f(x,y)&=&\sum_{(a,b):T_{23, 23,23}(a,b)=(x,y)} \frac{1}{\left|\mbox{Jac}(T_{23, 23,23}(a,b))\right|} f(a,b)\\
&=&\sum_{k=0}^{\infty}\frac{1}{(1+(1+k)x-y)^3}f\left(\frac{1}{1+(1+k)x-y}, \frac{1-x}{1+(1+k)x-y}\right).
\end{eqnarray*}

Of course, we would now like to calculate the transfer operator for the  TRIP map ${(e, e, (13))}$ of subsection \ref{(e, e, (13))}, but the calculations are a bit too cumbersome to actually write out.  However, note that the transfer operator will be in terms of Fibonacci numbers, and hence exponential in terms of the parameter $k$.

\section{Polynomial and Non-Polynomial Behavior in Transfer Operators}\label{poly}
\subsection{A Clean Criterion}
The transfer operator of the Gauss map $G:[0,1]\rightarrow [0,1]$ is 
$$\mathcal{L}_{G}(f)(x) = \sum_{k=1}^{\infty} \frac{1}{(k+x)^2} f\left( \frac{1}{k+x} \right).$$
Note that the denominator is a polynomial that is the square of a polynomial linear in $k$ and linear in $x$. The transfer operator for the original triangle map $T_{e,e,e}:\triangle \rightarrow \triangle$ is 
$$\mathcal{L}_{T_{e,e,e}}(f)(x,y) = \sum_{k=0}^{\infty} \frac{1}{(1+k x + y)^3} f\left(\frac{1}{(1+k x + y)},\frac{x}{(1+k x + y)} \right).$$
Now the denominator is a polynomial that is the cube of a polynomial linear in $k$ and linear in $x$ and $y$.  In the previous section, we have shown that the transfer operator for the TRIP map $((23),(23), (23))$ also has this behavior.  The goal of this section is to show that similar analogs hold for precisely 106 of the other triangle partition maps (and thus giving us polynomials behavior for 108 fo the maps), while the behavior of the transfer operator for the remaining 108 triangle maps is drastically different.

\begin{definition} The transfer operator for a triangle partition map $T_{\sigma, \tau_0,\tau_1}$ is said to have polynomial behavior if its Jacobian $\mbox{Jac}(T_{\sigma, \tau_0,\tau_1})$ is a polynomial in the variables $k$, $x$ and $y$, with possibly  a $(-1)^k$ term appearing.  Otherwise, we say that the transfer operator has non-polynomial behavior. 
\end{definition}
We will show that those transfer operators that have polynomial behavior  actually have their Jacobians equal to 
the cube of a polynomial that is linear in $k$ and linear in $x$ and $y$, with possibly  a $(-1)^k$ term appearing. 
In Appendix B of Part II, we explicitly list the 108 transfer operators that have polynomial behavior.  Those without this behavior are far too messy to list, a fact that we find interesting.

We will need the following well-known lemma, whose proof is a calculation:
\begin{lemma}

Consider a vector $v=(1,x,y)$, a matrix 
$M=
\begin{pmatrix}
a & b & c\\
d & e & f\\
g & h & i
\end{pmatrix},$
and the map 
$$\pi(vM)=\left(\frac{b+ex+hy}{a+dx+gy},\frac{c+fx+iy}{a+dx+gy}\right).$$
Then $$\mbox{Jac}(\pi(vM))=\frac{\det(M)}{(vM(1,0,0)^T)^3}.$$

\end{lemma}






\begin{theorem} \label{theorem:poly}A triangle partition map $T_{\sigma, \tau_0,\tau_1}$ has polynomial behavior if and only if the eigenvalues of the associated $F_1(\sigma, \tau_0,\tau_1)=\sigma F_1 \tau_1$ all have magnitude 1; it has non-polynomial-behavior otherwise. \end{theorem}

\begin{proof}
Our maps 
$T_{\sigma, \tau_0, \tau_1}$ are all infinite-to-one, but, for each nonegative integer $k$,  the map  is  one-to-one and onto when restricted to the subtriangle $\triangle_k(\sigma, \tau_0, \tau_1) $:
$$T_{\sigma, \tau_0, \tau_1}:\triangle_k(\sigma, \tau_0, \tau_1) \rightarrow \triangle.$$
For the rest of this proof, we will fix $k$, allowing us to consider the inverse $T_{\sigma, \tau_0, \tau_1}^{-1}.$ 

By the previous lemma and the definition of $T_{\sigma, \tau_0,\tau_1},$ it follows that $$\frac{1}{\mbox{Jac}(T_{\sigma, \tau_0,\tau_1}(a,b))}=\frac{det\left(M_{\sigma, \tau_0,\tau_1}\right)}{((1,x,y)M_{\sigma, \tau_0,\tau_1}(1,0,0)^{T})^{3}}$$ where $$M_{\sigma, \tau_0,\tau_1}=((V F_{0}^{-1}(\sigma, \tau_0,\tau_1) F_{1}^{-k}(\sigma, \tau_0,\tau_1)V^{-1})^{T})^{-1}.$$ 
Since $\det\left(M_{\sigma, \tau_0,\tau_1}\right)$ is real and has magnitude $\pm 1$ (since all matrices involved in its calculation also have real determinants of magnitude $\pm 1$), this implies that a given $T_{\sigma, \tau_0,\tau_1}$ is polynomial-behavior if and only if the first column of $M_{\sigma, \tau_0,\tau_1}$ contains terms polynomial in $k.$ 

For simplicity, write $M_{\sigma, \tau_0,\tau_1} = \left(C_1,C_2,C_3\right),$ where $C_i$ are the columns of $M_{\sigma, \tau_0,\tau_1}.$  
Note that the inverse of a triangle partition map, $T_{\sigma, \tau_0, \tau_1}^{-1},$ can be written as
\begin{eqnarray*}
T_{\sigma, \tau_0, \tau_1}^{-1}(x, y) 
&=& \pi\left((1, x, y)((V F_{0}^{-1}(\sigma, \tau_0,\tau_1) F_{1}^{-k}(\sigma, \tau_0,\tau_1)V^{-1})^{T})^{-1}      \right) \\
&=& \pi\left((1, x, y) M_{\sigma, \tau_0,\tau_1}\right) \\
&=& \left(\frac{(1,x,y)\cdot C_2}{(1,x,y)\cdot C_1},\frac{(1,x,y)\cdot C_3 }{(1,x,y)\cdot C_1}\right). \\
\end{eqnarray*}
Note that for any $(x,y) \in \bigtriangleup,$ $T_{\sigma, \tau_0, \tau_1}^{-1}(x, y)$ must be bounded as it must land back inside $\bigtriangleup.$ Since the inverse of each triangle partition map is bijective and since the choice of $k$ depends on which $\bigtriangleup_k$ the original $(x,y)$ lies in,  the first column of $M_{\sigma, \tau_0,\tau_1}$ must depend on $k.$ Thus, to show that the first column of $M_{\sigma, \tau_0,\tau_1}$ depends on $k$ polynomially, it is sufficient to show that $M_{\sigma, \tau_0,\tau_1}$ exhibits only polynomial dependence on $k$ -- for then, by the above argument, the first row of $M_{\sigma, \tau_0,\tau_1}$ must necessarily exhibit only polynomial dependence on $k.$

We have (suppressing the $(\sigma, \tau_0,\tau_1)$ in some of the matrices)
\begin{eqnarray*}
M_{\sigma, \tau_0,\tau_1}&=&((V F_{0}^{-1} F_{1}^{-k}V^{-1})^{T})^{-1} \\
&=& \left((V^{-1})^{T}((F_1^{-1})^k)^T(F_0^{-1})^{T}V^{T}\right)^{-1} \\
&=& (V^T)^{-1}((F_0^{-1})^T)^{-1}(((F_1^{-1})^k)^T)^{-1}((V^{-1})^T)^{-1} \\
&=& (V^T)^{-1}F_0^{T}(F_1^T)^{k}V^T.
\end{eqnarray*}

Let $A=(F_1)^{T}.$ We can find the Jordan decomposition of $A$; i.e. we can write it as 
$$A=P J P^{-1}$$
where $P$ is some invertible matrix of dimensions identical to those of $A,$ and $J$ is the Jordan normal form of $A$. Performing the Jordan decomposition for all 36 unique $A$s we see that there exist only six unique $J$s, namely 
$$J_1=
\begin{pmatrix}
1 && 1 && 0 \\
0 && 1 && 0 \\
0 && 0 && 1 \\
\end{pmatrix},$$
$$J_2=
\begin{pmatrix}
-1 && 0 && 0 \\
0 && 1 && 1 \\
0 && 0 && 1 \\
\end{pmatrix},$$ 
$$J_3=
\begin{pmatrix}
1 && 0 && 0 \\
0 && 1 && 1 \\
0 && 0 && 1 \\
\end{pmatrix},$$
$$J_4=
\begin{pmatrix}
1 && 0 && 0 \\
0 && \frac{1}{2}(1-\sqrt{5}) && 0 \\
0 && 0 && \frac{1}{2}(1+\sqrt{5}) \\
\end{pmatrix},$$
$$J_5=
D(roots(-1-t^2+t^3)),$$
and
$$J_6=
D(roots(-1-t+t^3)).$$
where $D(roots(-1-t^2+t^3))$ corresponds to a three-by-three square matrix with diagonal entries defined by the roots of $-1-t^2+t^3=0$; similarly for $D(roots(-1-t+t^3)).$

It can be shown that 
$$(J_1)^{k}=
\begin{pmatrix}
1 && k && 0 \\
0 && 1 && 0 \\
0 && 0 && 1 \\
\end{pmatrix}$$
$$(J_2)^{k}=
\begin{pmatrix}
(-1)^k && 0 && 0 \\
0 && 1 && k \\
0 && 0 && 1 \\
\end{pmatrix},$$ 
$$(J_3)^{k}=
\begin{pmatrix}
1 && 0 && 0 \\
0 && 1 && k \\
0 && 0 && 1 \\
\end{pmatrix},$$
$$(J_4)^{k}=
\begin{pmatrix}
1 && 0 && 0 \\
0 && (\frac{1}{2}(1-\sqrt{5}))^k && 0 \\
0 && 0 && (\frac{1}{2}(1+\sqrt{5}))^k \\
\end{pmatrix},$$
$$(J_5)^k=
D(roots(-1-t^2+t^3)^k),$$
and
$$(J_6)^k=
D(roots(-1-t+t^3)^k).$$

It is well-known that the diagonal elements of each $J$ are precisely the eigenvalues of the matrix $A$ from which it originated; further, $J_4$ through $J_6$ are diagonal and each contain at least one entry of magnitude greater than 1 on their respective diagonals. From this, it is clear that if $(F_1)^{T}(\sigma, \tau_0,\tau_1)$ (and hence $F_1(\sigma, \tau_0,\tau_1)$) has eigenvalues that all have magnitude 1, then the first column of $M_{\sigma, \tau_0,\tau_1}$ depends polynomially on $k,$ and hence the associated $T_{\sigma, \tau_0,\tau_1}$ has polynomial behavior. Otherwise, the first column of $M_{\sigma, \tau_0,\tau_1}$ depends exponentially on $k,$ and hence the associated $T_{\sigma, \tau_0,\tau_1}$ has non-polynomial behavior.

Assume $T_{\sigma, \tau_0,\tau_1}$ has polynomial behavior in $k.$ Then the first column of $M_{\sigma, \tau_0,\tau_1}$ must depend strictly polynomially on $k.$ From the above decomposition of $M_{\sigma, \tau_0,\tau_1},$ we see that the only place where $k$-dependance may enter the first column of $M_{\sigma, \tau_0,\tau_1}$ is through $A$; hence, by the above argument the eigenvalues of $F_1(\sigma, \tau_0,\tau_1)$ all have magnitude 1.  Now assume the eigenvalues of $F_1(\sigma, \tau_0,\tau_1)$ all have magnitude 1. Running above argument in reverse, we see that $T_{\sigma, \tau_0,\tau_1}$ must necessarily have polynomial behavior.
\end{proof}

\begin{corr}

A triangle partition map $T_{\sigma, \tau_0,\tau_1}$ is non-polynomial behavior if and only if the associated $M_{\sigma, \tau_0,\tau_1}$ depends exponentially on $k.$

\end{corr}

\begin{corr}
A triangle partition map $T_{\sigma, \tau_0,\tau_1}$ has polynomial behavior if and only if the associated $F_1(\sigma, \tau_0,\tau_1)=\sigma F_1 \tau_1$ is not diagonalizable.
\end{corr}

\begin{proof}
This follows immediately from the form of $J_1$ through $J_6$ in Theorem \ref{theorem:poly}.

\end{proof}

\begin{corr} 
Let $\sigma, \tau_0, \tau_1, \rho$ and $\gamma$ be three-by-three permutation matrices. If a triangle partition map $T_{\sigma, \tau_0,\tau_1}$ has polynomial behavior, then so does  the triangle partition map $T_{\rho\sigma, \gamma,\tau_1\rho^{-1}}$; similarly, if $T_{\sigma, \tau_0,\tau_1}$ have non-polynomial behavior, then so does $T_{\rho\sigma, \gamma,\tau_1\rho^{-1}}$.
\end{corr}

\begin{proof}
We need to compare $F_1 (\sigma, \tau_0,\tau_1)   $  and $F_1(\rho\sigma, \gamma,\tau_1\rho^{-1})$.
We have 
$$F_1(\rho\sigma, \gamma,\tau_1\rho^{-1}) = \rho\sigma F_1 \tau_1\rho^{-1}=\rho F_1(\sigma, \tau_0,\tau_1) \rho^{-1}.$$

Since conjugation will not change the Jordan canonical form of the relevant matrices, we are done.

\end{proof}

Note that transfer operators with polynomial behavior are basically traditional zeta functions (though not dynamical zeta functions).  In fact, we were initially tempted to call such transfer operators ``zeta''-like. 

Finally, what determines if a TRIP map $(\sigma, \tau_0, \tau_1)$ has polynomial behavior depends only on $\tau_1$.  By calculation, the 108 polynomial behavior TRIP maps are 
$$(\sigma, \tau_0, e), (\sigma, \tau_0, 12), (\sigma, \tau_0, 23).$$

\section{On Polynomial Behavior in Combination Triangle Partition Maps}
\label{originofpolycombo}

We are concerned with the form of the transfer operator $\mathcal{L}_{T}$ where 
$$\mathcal{L}_{T}f(x,y) = \sum_{(a,b):T(a,b)=(x,y)} \frac{1}{\left|\mbox{Jac}(T(a,b))\right|} f(a,b).$$
and $T$ is a finite composition of $n$ triangle partition maps defined by 
$$T=T_1 \circ T_2 \circ \dots \circ T_i \circ \dots \circ T_n$$
where each $T_i$ is a TRIP map for   a permutation triplet $(\sigma_i, (\tau_0)_i, (\tau_1)_i)$.

If the denominator of $\frac{1}{\mbox{Jac}(T(a,b))}$ is (non-trivially) polynomial in $k_i$ (also allowing for factors of $(-1)^{k_i}$) then $T$ gives rise to a transfer operator that has polynomial behavior in $k_i,$ and is itself polynomial behavior in $k_i;$ otherwise, $T$ gives rise to a transfer operator that has  non-polynomial behavior in $k_i,$ and has itself non-polynomial behavior. Here we state and prove a theorem regarding the polynomial behavior of combo triangle partition maps, and present some corollaries. To avoid redundancy, we will refer to polynomial dependence on $k_i$ aside from factors of $(-1)^{k_i}$ as polynomial dependence on $k_i$ or strictly polynomial dependence on $k_i.$

\begin{theorem} A combination triangle partition map $T=T_1 \circ T_2 \circ \dots \circ T_i \circ \dots \circ T_n$ has polynomial behavior in $k_i$ if and only if the eigenvalues of the associated $F_{1_i}$ (the $F_1$ matrix corresponding to $T_i$) all have magnitude 1; it has non-polynomial behavior in $k_i$ otherwise. \end{theorem}
\begin{proof}
We have already shown explicitly that $\frac{1}{\mbox{Jac}(T(a,b))}=\frac{det\left(M\right)}{((1,x,y)M(1,0,0)^{T})^{3}}$ where $$M=M_1 M_2 \dots M_i \dots M_n$$
and
$$M_i=((V F_{0}^{-1} (\sigma_i, (\tau_0)_i, (\tau_1)_i)F_{1_i}^{-k}(\sigma_i, (\tau_0)_i, (\tau_1)_i)V^{-1})^{T})^{-1}.$$  Since $det\left(M\right)$ is real and has magnitude $\pm 1$, this implies that $T$ has polynomial behavior in $k_i$ if and only if the first column of $M$ contains terms polynomial in $k_i$.
But now the argument will mirror the proof of theorem \ref{theorem:poly}, since what matters is the nature of each $F_{1_i}^{-k}(\sigma_i, (\tau_0)_i, (\tau_1)_i)$, right down to there only being six unique Jordan canonical forms.

\end{proof}

We have the following corollaries, whose proofs mimic the corollaries of theorem \ref{theorem:poly}.

\begin{corr}

A combination triangle partition map $T$ as defined in the above theorem is non-polynomial-behavior in $k_i$ if and only if the associated $M$ depends exponentially on $k_i.$

\end{corr}


\begin{corr}
A combination triangle partition map $T$ as defined in the above theorem is polynomial-behavior in $k_i$ if and only if the associated $F_{1_i}$ is not diagonalizable; it is non-polynomial-behavior otherwise.
\end{corr}

\section{Conclusion}
While we find the distinction between polynomial versus non-polynomial behavior interesting, the real goal of this paper is to set-up the machinery for our paper ``Functional Analysis behind a Family of Multidimensional Continued Fractions: Part II.''  In this next paper we will find
 eigenfunctions of eigenvalue 1 for transfer operators associated with 18  polynomial-behavior triangle partition maps  The formidably complex form of the non-polynomial-behavior transfer operators and lack of appropriate techniques makes finding leading eigenfunctions prohibitively difficult; however, finding leading eigenfunctions for the remaining 90 polynomial-behavior transfer operators appears a doable, though computationally-intensive, problem. We will then show  the nuclearity of transfer operators, thought of as acting on one-dimensional families of Hilbert spaces, associated with 36 polynomial-behavior maps. 
 
 In the conclusion of the next paper, we further discuss many of natural and important questions that remain open.

\end{document}